\documentclass[12pt,leqno]{article}
\usepackage{amssymb,amsmath,amsthm}

\newtheorem{lem}{Lemma}
\newtheorem{theorem}[lem]{Theorem}
\newtheorem{cor}[lem]{Corollary}

\newtheorem{prop}[lem]{Proposition}

 \newcommand{\Q}{\ensuremath{\mathbb{Q}}}
 
 \newcommand{\C}{\ensuremath{\mathbb{C}}}

\newcommand{\Gal}{{\rm Gal}}

\renewcommand{\Re}{{\rm Re \,}}

\begin{document}
\title{\bf Independence of  Artin L-functions}
\author{Mircea Cimpoea\c s\\Florin Nicolae}

\maketitle

\begin{abstract}
Let $K/\Q$ be a finite Galois extension. Let $\chi_1,\ldots,\chi_r$ be $r\geq 1$ distinct characters of the Galois 
group with the associated Artin L-functions $L(s,\chi_1),\ldots, L(s,\chi_r)$. Let $m\geq 0$.
We prove that the derivatives $L^{(k)}(s,\chi_j)$, $1\leq j\leq r$, $0\leq k\leq m$, are linearly independent over
the field of meromorphic functions of order $<1$. From this it follows that the L-functions corresponding to the
irreducible characters are algebraically independent over the field of meromorphic functions of order $<1$.


\end{abstract}

\section*{Introduction}

\indent Let $K/\Q$ be a finite Galois extension. For the character $\chi$ of a 
representation of the Galois group $G:=\Gal(K/\Q)$ on a finite dimensional 
complex vector space, let $L(s, \chi, K/\Q)$ be the 
corresponding Artin $L$-function (\cite{artin}, P. 296). It was proved in \cite[Theorem 1]{lucrare}
that the derivatives of any order of 
Artin $L$-functions to finitely many distinct characters of $G$ are linearly 
independent over $\mathbb C$. A more general result concerning linear independence 
over $\mathbb C$ of functions in a class which contains the Artin $L$-functions was proved in \cite{kac}. In \cite{mol} it was proved 
the independence of any family of suitable $L$-functions (including Artin $L$-functions) with respect to a ring generated by a slow varying function.\\
\indent For $\varepsilon>0$ let
$$\mathcal F_{\varepsilon}=\{f:(1+\varepsilon, +\infty)\to \C\}, $$
$$\mathcal B_{\varepsilon}=\{f\in \mathcal F_{\varepsilon}\;:\; \forall a>0 \;
\lim_{\sigma\to +\infty}e^{-a\sigma}|f(\sigma)|=0 \},$$
$$\mathcal V_{\varepsilon}=\{f\in \mathcal F_{\varepsilon}\;:\exists a>0\;\lim_{\sigma \to +\infty}e^{a\sigma}|f(\sigma)|=0 \}.$$
The main result of this paper is \\
\textbf{Theorem 7.} \emph{Let $K/\mathbb Q$ be a finite Galois extension, and let $\chi_1,\ldots,\chi_r$ be distinct characters of the Galois 
group with the associated Artin L-functions $L(s,\chi_1),\ldots, L(s,\chi_r)$.
Let $\varepsilon>0$. Let $\mathcal A_{\varepsilon}\subset \mathcal B_{\varepsilon}$ be a $\mathbb C$-vector space 
with $$\mathcal A_{\varepsilon}\cap  \mathcal V_{\varepsilon} = \{0\}.$$ Let  $m\geq 0$. If the functions 
 $G_{jk}(\sigma)\in \mathcal A_{\varepsilon}$ satisfy 
$$ \sum_{j=1}^r \sum_{k=0}^m G_{jk}(\sigma)L^{(k)}(\sigma,\chi_j) = 0,\; \sigma>1+\varepsilon, $$
then $G_{jk}=0$, $1\leq j\leq r$, $0\leq k\leq m$. }\\
As a consequence we have
\\
\textbf{Corollary 8.}\emph{
The functions $L^{(k)}(s,\chi_j)$, $1\leq j\leq r$, $0\leq k\leq m$ are linearly independent over
the field of meromorphic functions of order $<1$.}\\
Since the Artin L-functions are meromorphic of order $1$ this result is best possible when we look for linear dependence with coefficients 
meromorphic functions. This extends the main result of \cite{lucrare}. Let $\chi_1,\ldots,\chi_h$ be the irreducible characters of the Galois 
group $Gal(K/\mathbb Q)$. In \cite[Corollary 4]{lucrare} it was proved that
the Artin L-functions $L(s,\chi_1), \ldots, L(s,\chi_h)$ are algebraically independent over $\mathbb C$.
This extended Artin's result \cite[Satz 5, P. 106]{artin1} that $L(s,\chi_1), \ldots$ , $L(s,\chi_h)$ are multiplicatively independent.
In Corollary $9$ we prove that $L(s,\chi_1), \ldots, L(s,\chi_h)$ are algebraically independent over 
the field of meromorphic functions of order $<1$.

\section{Arithmetic Functions and Dirichlet Series}

In this section we present some properties of Dirichlet series which are needed for the proof of the results on Artin L-functions in section 
$2$.
 
\begin{lem}
Let $(a_n(\sigma))_{n\geq 1}$ be a sequence of functions in $(\mathcal F_\varepsilon\setminus
\mathcal V_\varepsilon) \cup \{0\}$ such that 
there exists a function
$M:(1+\varepsilon, +\infty)\to [0,+\infty)$ with the properties:
\begin{enumerate}
 \item[(i)] $n^{-\varepsilon}|a_n(\sigma)|\leq M(\sigma), \; \forall n\geq 1,\sigma>1+\varepsilon$,
 \item[(ii)] For all $a>0$ it holds that $\lim_{\sigma\to +\infty}e^{-a\sigma}M(\sigma)=0$.
\end{enumerate}
Then
\begin{enumerate}
\item[(1)] The series of functions $F(\sigma):=\sum_{n=1}^{\infty}
\frac{a_n(\sigma)}{n^\sigma}$ is absolutely convergent on $(1+\varepsilon, +\infty)$.

\item[(2)] If $F(\sigma)$ is
identically zero then $a_n(\sigma)$ is identically zero for any $n\geq 1$.
\end{enumerate}
\end{lem}

\begin{proof}
(1) Let $\sigma>1+\varepsilon$. By $(i)$, it holds that
$$ \left|\frac{a_n(\sigma)}{\sigma^n}\right|  \leq \frac{M(\sigma)}{n^{\sigma-\varepsilon}}.$$
Since the series $\sum_{n=1}^{\infty}\frac{1}{n^{\sigma-\varepsilon}}$ is convergent, it follows that
the series $\sum_{n=1}^{\infty}
\frac{a_n(\sigma)}{n^\sigma}$ is absolutely convergent.

(2) We show inductively that $a_n$ is identically zero for any $n\geq 1$. Suppose that $k=1$ or that $k>1$ and, for induction, $a_1,\ldots,a_{k-1} $ 
are identically zero. For $\sigma>1+\varepsilon$ we have that
$$ |a_k(\sigma)| = \left|k^\sigma \sum_{n=k+1}^{\infty} \frac{a_n(\sigma)}{n^\sigma} \right|
\leq \sum_{n=k+1}^{\infty} \frac{|a_n(\sigma)|k^{\sigma}}{n^{\sigma}}\leq $$
\begin{equation}
\leq M(\sigma) \sum_{n=k+1}^{\infty} \frac{k^{\sigma}}{n^{\sigma-\varepsilon}} = 
k^{\varepsilon}M(\sigma)
\sum_{n=k+1}^{\infty}\frac{1}{(\frac{n}{k})^{\sigma-\varepsilon}}.
\end{equation}
We choose $\delta > \varepsilon$. Let  $\sigma \geq 1+\delta$. From $(1)$ it follows that
$$
|a_k(\sigma)| \leq k^{\varepsilon}M(\sigma)
\sum_{n=k+1}^{\infty}\frac{1}{(\frac{n}{k})^{\frac{\sigma+1-\varepsilon}{2}}}
\frac{1}{(\frac{n}{k})^{\frac{\sigma-1-\varepsilon}{2}}} \leq 
$$
\begin{equation}
\leq k^{\varepsilon}M(\sigma)
\left(\sum_{n=k+1}^{\infty}\frac{1}{(\frac{n}{k})^{1+\frac{\delta-\varepsilon}{2}}}\right)
\frac{1}{(\frac{k+1}{k})^{\frac{\sigma-1-\varepsilon}{2}}}.
\end{equation}
Let $0<a<\frac{1}{2}(\log(k+1)-\log k)$. From $(2)$ it follows that 
\begin{equation}
 e^{a\sigma}|a_k(\sigma)| \leq k^{\frac{\varepsilon-1}{2}} (k+1)^{\frac{1+\varepsilon}{2}} \left(\sum_{n=k+1}^{\infty}
\frac{1}{\left( \frac{n}{k} \right)^{1+\frac{\delta-\varepsilon}{2}}}\right) e^{(a-\frac{\log (k+1)-\log k}{2})\sigma}M(\sigma).  
\end{equation}
From $(3)$ and hypothesis $(ii)$ it follows that
$$\lim_{\sigma \rightarrow +\infty}e^{a\sigma}|a_k(\sigma)|=0,$$
so $a_k \in\mathcal V_{\varepsilon}$. Since, by hypothesis, $a_k\in (\mathcal F_\varepsilon\setminus
\mathcal V_\varepsilon) \cup \{0\}$ it follows that $a_k=0$.  
\end{proof}

Let $\varepsilon>0$. If $f(n)$ is an arithmetic function of order
$O(n^{\varepsilon})$, i.e. there exists $C>0$ such that $|f(n)|\leq Cn^{\varepsilon}$ for every $n\geq 1$ , then
the associated Dirichlet series
$$F(s):=\sum_{n=1}^{\infty}\frac{f(n)}{n^s}$$
defines a holomorphic function in the half plane $\Re(s) > 1 + \varepsilon$. 
The \emph{$k$-th arithmetic derivative} of $f(n)$ is
$$f^{(k)}(n) = (-1)^k f(n) \log^k n.$$
It holds that the \emph{$k$-th derivative} of $F(s)$ is the Dirichlet series
$$F^{(k)}(s)= \sum_{n=1}^{\infty} \frac{f^{(k)}(n)}{n^s}.$$

\begin{theorem}
Let $\varepsilon>0$.
Let $f_1(n),\ldots,f_r(n)$ be arithmetic functions of order
$O(n^{\varepsilon})$ linearly independent over $\mathbb C$ with the associated Dirichlet series $F_1(s),\ldots,F_r(s)$. 
Let $\mathcal A_{\varepsilon}\subset \mathcal B_{\varepsilon}$ be a $\mathbb C$-vector space 
with $$\mathcal A_{\varepsilon}\cap  \mathcal V_{\varepsilon} = \{0\}.$$  
Let $G_1(\sigma),\ldots,G_r(\sigma)\in \mathcal A_{\varepsilon}$
such that 
\begin{equation}
\sum_{j=1}^r G_{j}(\sigma) F_j(\sigma)=0,\; \sigma>1+\varepsilon.
\end{equation}
Then $G_1=\cdots=G_r=0$.  
\end{theorem}

\begin{proof}
Let $\sigma>1+\varepsilon$. From $(4)$ it follows that
\begin{equation}
\sum_{n=1}^{\infty} \left(\sum_{j=1}^r G_{j}(\sigma)f_j(n) \right)n^{-\sigma} = 0.
\end{equation}
For $n\geq 1$ let
\begin{equation}
a_n(\sigma):=\sum_{j=1}^r G_{j}(\sigma)f_j(n).
\end{equation}
The functions $a_n(\sigma)$ are contained in the space $\mathcal A_{\varepsilon}$.
Since the functions $f_j(n)$ have order $O(n^{\varepsilon})$, there exists a
constant $C>0$ such that 
\begin{equation}
\max\{|f_j(n)|\;:\; 1\leq j\leq r\} \leq C n^{\varepsilon},\;
\forall n\geq 1.
\end{equation}
Let
\begin{equation}
M(\sigma) := C \sum_{j=1}^r |G_{j}(\sigma)|.
\end{equation}
From $(6),(7)$ and $(8)$ it follows that
\begin{equation}
|a_n(\sigma)| \leq M(\sigma)n^{\varepsilon},\; \forall n\geq 1.
\end{equation}
Also, since $G_{j} \in \mathcal B_{\varepsilon}$, from $(6)$ and $(9)$ it follows that
$$ \lim_{\sigma \rightarrow +\infty}e^{-a\sigma}M(\sigma) =
0,\; \forall a>0,$$
hence the conditions of the Lemma $1$ are satisfied.
From $(5)$ and  Lemma $1$ it follows that
\begin{equation}
a_n(\sigma)=0,\; \forall n\geq 1. 
\end{equation}
Since the functions $f_1(n),\ldots,f_r(n)$ are linearly
independent over $\mathbb C$, from $(6)$ and $(10)$ it follows that $G_{j}(\sigma)=0$, $1\leq j\leq r$, as required.
\end{proof}

An arithmetic function $f(n)$ is called \emph{multiplicative}, if
$f(1)=1$ and
$$ f(nm)=f(n)f(m),\; n,m\in\mathbb N\;\text{with }\gcd(n,m)=1.$$
Two multiplicative arithmetical functions $f(n)$ and $g(n)$ are called
\emph{equivalent} (see \cite{kac1}) if 
$f(p^j)=g(p^j)$ for all integers $j\geq 1$ and all but finitely many
primes $p$. Let $e(n)$ be the identity function,
defined by $e(1)=1$ and $e(n)=0$ for $n\geq 2$. We recall the
following result of Kaczorowski, Molteni and Perelli \cite{kac}.

\begin{lem}(\cite[Lemma 1]{kac})
Let $f_1(n),\ldots,f_r(n)$ be multiplicative functions such that
$e(n),f_1(n),\ldots,f_r(n)$ are pairwise non-equivalent, and let
$m$ be a non-negative integer. Then the functions
$$f_1^{(0)}(n),\ldots,f_1^{(m)}(n), f_2^{(0)}(n),\ldots,f_2^{(m)}(n),
\ldots, f_r^{(0)}(n),\ldots,f_r^{(m)}(n)$$
are linearly independent over $\mathbb C$.
\end{lem}

\begin{cor}
Let $\varepsilon>0$. Let $f_1(n),\ldots,f_r(n)$ be multiplicative functions of order $O(n^{\varepsilon})$
such that $e(n),f_1(n),\ldots$, $f_r(n)$ are pairwise non-equivalent. 
Let $$F_j(s)=\sum_{n=1}^{\infty}\frac{f_j(n)}{n^s},\; j=1,\ldots,r,$$
be the associated Dirichlet series. Let $\mathcal A_{\varepsilon}\subset \mathcal B_{\varepsilon}$ be a $\mathbb C$-vector space 
with $$\mathcal A_{\varepsilon}\cap  \mathcal V_{\varepsilon} = \{0\}.$$ Let $m\geq 0$. If the functions 
 $G_{jk}(\sigma)\in \mathcal A_{\varepsilon} $ satisfy
$$ \sum_{j=1}^r \sum_{k=0}^m G_{jk}(\sigma)F^{(k)}_j(\sigma) = 0,\; \sigma>1+\varepsilon, $$
then $G_{jk}=0$, $1\leq j\leq r$, $0\leq k\leq m$.  
\end{cor}

\begin{proof}
If follows from Lemma $3$ and Theorem $2$.
\end{proof}

\begin{prop}
 Let $f:\mathbb C\rightarrow \mathbb C$ be an entire function of order $\rho<1$ which is not identically zero. Let $\varepsilon>0$. Then 
$f|_{(1+\varepsilon,+\infty)}\in  \mathcal B_{\varepsilon}\setminus \mathcal V_{\varepsilon}$.
\end{prop}

\begin{proof}
Let $\rho<\lambda<1$. There exists a constant $C>0$ such that
$$ |f(s)|\leq C e^{|s|^{\lambda}},\; s\in \mathbb C.$$
Let $a>0$ and $\sigma>1+\epsilon$. We have that
$$ e^{-a\sigma}|f(\sigma)|\leq Ce^{-a\sigma+\sigma^{\lambda}}$$ so
$$ \lim_{\sigma\to +\infty} e^{-a\sigma}|f(\sigma)| = 0,$$
hence $f|_{(1+\varepsilon,+\infty)}\in  \mathcal B_{\varepsilon}$. 
If $f$ is polynomial then $$\lim_{\sigma \to +\infty}e^{a\sigma}|f(\sigma)|=+\infty,\; \forall a>0,$$
so $f|_{(1+\varepsilon,+\infty)}\notin  \mathcal V_{\varepsilon}$. If $f$ is not polynomial, then by Hadamard's Theorem, there exists $A\neq 0$ such that
\begin{equation}
 f(s)=As^m \prod_{n=1}^{\infty}(1-\frac{s}{s_n}),s\in \mathbb C, 
\end{equation}
where $m$ is the multiplicity of $s_0=0$ as zero of $f$, and $s_1,s_2,\ldots$ are the non-zero zeros of $f$.
Let $$E(s):=\prod_{n=1}^{\infty}(1-\frac{s}{s_n}).$$
It is known (see \cite[Ch. 5, Corollary 5.4]{stein}) that there exists a stricly increasing sequence $(r_k)_{k\geq 1}$ of 
positive numbers with $\lim_{k\to+\infty}r_k=+\infty$ and a constant $B>0$ such that 
\begin{equation}
|E(r_k)|\geq e^{-Br_k^{\lambda}},\; \forall k\geq 1.
\end{equation}
Let $a>0$. From $(11)$ and $(12)$ it follows that
$$ e^{ar_k}|f(r_k)| = e^{ar_k}|A|r_k^m |E(r_k)|\geq |A|r_k^m e^{ar_k-Br_k^{\lambda}} \rightarrow +\infty,$$
hence $f|_{(1+\varepsilon,+\infty)}\notin  \mathcal V_{\varepsilon}$.
\end{proof}

\begin{cor}
With the assumptions of Corollary $4$, the holomorphic functions $F^{(k)}_j(s)$, $1\leq j\leq r$, $0\leq k\leq m$ are linearly independent over
the field of meromorphic functions of order $<1$.
\end{cor}

\begin{proof}
 Suppose that there exists a linear combination
$$\sum_{j=1}^r \sum_{k=0}^m Q_{jk}(s)F^{(k)}_j(s) = 0,\;\Re s>1+\epsilon, $$
where $Q_{jk}$ are meromorphic functions of order $<1$. It is known that a meromorphic function of order $<1$ is
a quotient of entire functions of order $<1$. So we may suppose that $Q_{jk}$ are entire functions of order $<1$. 
Let $\sigma>1+\epsilon$. We have that 
$$\sum_{j=1}^r \sum_{k=0}^m Q_{jk}(\sigma)F^{(k)}_j(\sigma) = 0.$$
From Proposition $5$ and Corollary $4$ it follows that 
$$Q_{jk}(\sigma)=0,\;\sigma>1+\epsilon,$$
hence $$Q_{jk}(s)=0,\;s\in\mathbb C,$$
by the identity principle for holomorphic functions.
\end{proof}

\section{Artin L-functions}

Let $K/\mathbb Q$ be a finite Galois extension.
It was proved in \cite[Theorem 1]{lucrare} that the derivatives of any order of 
Artin $L$-functions to finitely many distinct characters of the Galois group are linearly 
independent over $\mathbb C$. In our main result we extend this:

\begin{theorem}
Let $K/\mathbb Q$ be a finite Galois extension, and let $\chi_1,\ldots,\chi_r$ be distinct characters of the Galois 
group with the associated Artin L-functions $L(s,\chi_1),\ldots, L(s,\chi_r)$.
Let $\varepsilon>0$. Let $\mathcal A_{\varepsilon}\subset \mathcal B_{\varepsilon}$ be a $\mathbb C$-vector space 
with $$\mathcal A_{\varepsilon}\cap  \mathcal V_{\varepsilon} = \{0\}.$$ Let  $m\geq 0$. If the functions 
 $G_{jk}(\sigma)\in \mathcal A_{\varepsilon}$ satisfy 
$$ \sum_{j=1}^r \sum_{k=0}^m G_{jk}(\sigma)L^{(k)}(\sigma,\chi_j) = 0,\; \sigma>1+\varepsilon, $$
then $G_{jk}=0$, $1\leq j\leq r$, $0\leq k\leq m$. 
\end{theorem}

\begin{proof}
Let $$L(s,\chi_j)=\sum_{n=1}^{\infty} \frac{f_j(n)}{n^s},\; j=1,\ldots,r$$
be the Dirichlet series expansion of $L(s,\chi_j)$ in the half-plane $\Re s>1$.
Since in $\Re s>1$ the function $L(s,\chi_j)$ is defined by an Euler product, the arithmetic function $f_j(n)$
is multiplicative. It is well known that for Artin L-functions, the function $f_j(n)$ is $O(n^{\delta})$ for any $\delta>0$,
hence $O(n^{\varepsilon})$.
We show that $e(n),f_1(n),\ldots,f_r(n)$ are pairwise non-equivalent.
For a prime number $p$ which is not ramified in $K$, the value $f_j(p)$ equals the value of the character 
$\chi_j$ on the Frobenius class associated to $p$. For distinct characters $\chi_j$ and $\chi_k$ there are,
by Chebotarev's density theorem, infinitely many primes $p$ such that $f_j(p)\neq f_k(p)$,
so the arithmetic functions $f_j$ and $f_k$ are not equivalent in the sense of section $1$.
Also, $e(p)=0$ for any prime $p$, while there exist infinitely many non-ramified primes $p$ with
$f_j(p)= \chi_j(1) \neq 0$, hence the arithmetic functions $e$ and $f_j$ are not equivalent.
We apply Corollary $4$.
\end{proof}

\begin{cor}
Let $K/\mathbb Q$ be a finite Galois extension. Let $\chi_1,\ldots,\chi_r$ be $r\geq 1$ distinct characters of the Galois 
group with the associated Artin L-functions $L(s,\chi_1),\ldots, L(s,\chi_r)$.
Let $m\geq 0$. The meromorphic functions $L^{(k)}(s,\chi_j)$, $1\leq j\leq r$, $0\leq k\leq m$ are linearly independent over
the field of meromorphic functions of order $<1$.
\end{cor}

\begin{proof}
 Apply Theorem $7$ and Corollary $6$.
\end{proof}

\noindent Let $\chi_1,\ldots,\chi_h$ be the irreducible characters of the Galois 
group. In \cite[Corollary 4]{lucrare} it was proved that
the Artin L-functions $L(s,\chi_1), \ldots, L(s,\chi_h)$ are algebraically independent over $\mathbb C$.
This extended Artin's result \cite[Satz 5, P. 106]{artin1} that $L(s,\chi_1), \ldots$ , $L(s,\chi_h)$ are multiplicatively independent. 
Now we can prove more:

\begin{cor}
Let $K/\mathbb Q$ be a finite Galois extension, and let $\chi_1,\ldots,\chi_h$ be the irreducible characters of the Galois 
group. Then the Artin L-functions $L(s,\chi_1), \ldots, L(s,\chi_h)$ are algebraically independent over 
the field of meromorphic functions of order $<1$.
\end{cor}

\begin{proof}

This follows from Corollary $8$ and the fundamental property 
$$ L(s,\chi_1)^{n_1}\cdots L(s,\chi_h)^{n_h} = L(s,n_1\chi_1 + \cdots + n_h\chi_h). $$

\end{proof}

\begin{cor}
Let $K_1/\mathbb Q,\ldots,K_r/\mathbb Q$ be $r\geq 1$ distinct finite Galois extensions
with the Dedekind zeta-functions $\zeta_{K_1},\ldots,\zeta_{K_r}$. Let $m\geq 0$. The functions 
$$\zeta^{(0)}_{K_1},\ldots,\zeta^{(m)}_{K_1},\zeta^{(0)}_{K_2},\ldots,\zeta^{(m)}_{K_2},\ldots,\zeta^{(0)}_{K_r}, \ldots, \zeta^{(m)}_{K_r}$$
are linearly independent over the field of meromorphic functions of order $<1$.
\end{cor}

\begin{proof}
As in the proof of \cite[Corollary 5]{lucrare}, it holds that
$$\zeta_{K_j}(s)=L(s,\chi_j,K/\mathbb Q),\;j=1,\ldots,r, $$
where $K:=K_1\cdots K_r$ is the compositum of the fields $K_1,\ldots,K_r$ and $\chi_1,\ldots,\chi_r$
are distinct characters of the Galois group of $K$. We apply Corollary $8$.
\end{proof}

{}

\vspace{2mm} \noindent {\footnotesize
\begin{minipage}[b]{15cm}
Mircea Cimpoea\c s, Simion Stoilow Institute of Mathematics, Research unit 5, P.O.Box 1-764,\\
Bucharest 014700, Romania, E-mail: mircea.cimpoeas@imar.ro
\end{minipage}}

\vspace{2mm} \noindent {\footnotesize
\begin{minipage}[b]{15cm}
Florin Nicolae, Simion Stoilow Institute of Mathematics, P.O.Box 1-764,\\
Bucharest 014700, Romania, E-mail: florin.nicolae@imar.ro
\end{minipage}}
\end{document}